\pdfoutput=1
\documentclass[12pt,a4paper]{amsart}
\usepackage[utf8]{inputenc}
\usepackage[T1]{fontenc}
\usepackage[british]{babel}
\usepackage[babel=true,tracking=true]{microtype}
\usepackage[margin=3cm]{geometry}
\usepackage{amsmath,amsthm,amssymb,bookmark,csquotes,nicematrix}
\usepackage[backend=biber,style=alphabetic,giveninits]{biblatex}
\DeclareNameAlias{default}{family-given}
\AtBeginBibliography{\small}

\theoremstyle{plain}
\newtheorem{theorem}{Theorem}[section]
\newtheorem{lemma}[theorem]{Lemma}
\newtheorem{corollary}[theorem]{Corollary}
\newtheorem{proposition}[theorem]{Proposition}

\theoremstyle{definition}
\newtheorem{definition}[theorem]{Definition}
\newtheorem{problem}[theorem]{Problem}
\newtheorem{example}[theorem]{Example}
\newtheorem{question}[theorem]{Question}

\theoremstyle{remark}
\newtheorem{remark}[theorem]{Remark}
\newtheorem{notation}[theorem]{Notation}
\newtheorem{computation}[theorem]{Computation}
\newtheorem*{conventions}{Notations and conventions}
\newtheorem*{acknowledgements}{Acknowledgements}
\newtheorem*{structure}{Structure of the paper}

\title[Weighted Grassmannians and their explicit description]{Weighted Grassmannians and \\ their explicit description}
\author{Mikhail Ovcharenko}
\address
{
  \textnormal{Steklov Mathematical Institute of RAS, 8 Gubkina street, Moscow 119991, Russia.}
  \newline
  \textnormal{HSE University, Laboratory of Mirror Symmetry, 6 Usacheva str., Moscow 119048, Russia.}
}
\email{michael.a.ovcharenko@gmail.com, ovcharenko@mi-ras.ru}

\DeclareMathOperator{\Aut}{Aut}
\DeclareMathOperator{\Cl}{Cl}
\DeclareMathOperator{\codim}{codim}
\DeclareMathOperator{\coker}{coker}
\DeclareMathOperator{\depth}{depth}
\DeclareMathOperator{\End}{End}
\DeclareMathOperator{\GL}{GL}
\DeclareMathOperator{\Gr}{Gr}
\DeclareMathOperator{\Hom}{Hom}
\DeclareMathOperator{\height}{ht}
\DeclareMathOperator{\Id}{Id}
\DeclareMathOperator{\im}{im}
\DeclareMathOperator{\Lie}{Lie}
\DeclareMathOperator{\Proj}{Proj}
\DeclareMathOperator{\PGL}{PGL}
\DeclareMathOperator{\rk}{rk}
\DeclareMathOperator{\SL}{SL}
\DeclareMathOperator{\SO}{SO}
\DeclareMathOperator{\Sing}{Sing}
\DeclareMathOperator{\Spec}{Spec}
\DeclareMathOperator{\Stab}{Stab}
\DeclareMathOperator{\Sym}{Sym}

\begin{document}

\begin{abstract}
  We propose an explicit construction of a weighted generalised Grassmannian. For a weighted Grassmannian (i.e., for series A) we obtain an effective parametrisation of possible \(\mathbb{Z}\)-gradings on Pl\"{u}cker coordinates, and provide the explicit formulae for its dualising sheaf and Hilbert series in terms of this parametrisation. Our approach can be generalised to other irreducible root systems.
\end{abstract}

\maketitle

\section{Introduction}\label{section:introduction}

In this paper we give an explicit description of weighted Grassmannians: possible \(\mathbb{Z}\)-gradings on Pl\"{u}cker coordinates, dualising sheaf, singular locus, and Hilbert series.

Complete intersections in projective spaces are standard examples of projective varieties whose geometric properties can be expressed in combinatorial terms. For example, if \(X \subset \mathbb{P}^n\) is a smooth complete intersection of hypersurfaces of degrees \((d_1, \ldots, d_c)\), then its anticanonical sheaf \(\omega_X^{\vee}\) is isomorphic to \(\mathcal{O}_X(n + 1 - \sum_{j = 1}^c d_j)\). In particular, if \(\sum_{j = 1}^c d_j < n + 1\), then \(X\) is a \emph{Fano variety}, i.e., \(\omega_X^{\vee}\) is ample: such varieties constitute a fundamental part of the classification of projective varieties.

It is natural to replace the ambient space of a complete intersection with another (possibly, singular) Fano variety with similar geometric properties, and consider complete intersections therein. Two main sources of such ambient varieties are \emph{weighted projective spaces} and \emph{generalised Grassmannians}. We refer the reader to~\cite{przyjalkowski/weighted} for preliminaries on weighted projective spaces, and to~\cite[Section~2]{munoz/campana} for a survey of rational homogeneous spaces. Smooth Fano varieties naturally appear in this way, the idea goes back to the original work of Mukai (see~\cites{debiase/fano,mukai/fano}).

Geometric properties of complete intersections in a weighted projective
space are determined by the combinatorics of weights and degrees (for example, see~\cites{chen/quasismooth,pizzato/nonvanishing,przyjalkowski/automorphisms,przyjalkowski/bounds}). In contrast, geometric properties of a complete intersection in a generalised Grassmannian depend on the representation theory of a simple linear algebraic group (for example, see~\cites{konno/torelli-I,konno/torelli-II,konno/torelli-III}).

One may wonder if these approaches can be mixed, i.e., can we introduce weights in the representation-theoretic construction of a generalised Grassmannian. The work on this matter was started by Corti--Reid and Grojnowski (see~\cite{corti/weighted}), and was followed by Qureshi, Szendr\H{o}i et al. (for example, see~\cites{qureshi/flag-I,qureshi/threefolds,qureshi/flag-II}).
The main idea of their approach can be described as follows.

Let \(G\) be a reductive linear algebraic group, and \(P \subset G\) be a maximal parabolic subgroup. We denote by \(Y = G / P\) the associated generalised Grassmannian. It admits a minimal \(G\)-equivariant embedding \(Y \hookrightarrow \mathbb{P}\), where \(\psi \colon G \hookrightarrow \GL(W)\) is the corresponding fundamental representation. Its affine cone \(\widehat{Y} \subset \mathbb{A}(W)\) admits a natural action of both \(G\) and \(Z(\GL(W)) \simeq \mathbb{G}_m\), where we denote by \(\mathbb{A}(W)\) the associated affine space of \(W\). Let \(\mu^{\vee}\) be a (multiplicative) coweight of \(G\), and
\[
  \tau^{\vee}_W \colon \mathbb{G}_m \hookrightarrow \GL(W), \quad
  \tau^{\vee}_W(c) = c \cdot \Id_W \in \GL(W),
\]
be the central coweight. Then \(\gamma^{\vee} = \psi(\mu^{\vee}) + m \cdot \tau^{\vee}_W\) defines a \(\mathbb{Z}_{+}\)-grading on the vector space \(W\) for \(m \gg 0\) (note that the second summand is necessary if \(G\) is simple). The corresponding GIT quotient \(\widehat{Y} // \mathbb{G}_m\) by the coweight \(\gamma^{\vee}\) is usually called a \emph{weighted generalised Grassmannian} (see~\cite[Sections~2--3]{qureshi/flag-I}).

In particular, if \(\mu^{\vee} = 0\), then we obtain usual generalised Grassmannians. Another example is a weighted projective space: it arises from the standard representation of \(G = \GL(V)\). Note that we cannot replace \(\GL(V)\) with the action of \(\SL(V) \times \mathbb{G}_m\): in this way we obtain only weighted projective spaces \(\mathbb{P}(a_0, \ldots, a_n)\) such that \(\sum_{i = 0}^n a_i\) is divisible by \(\dim(V) = n + 1\). In~\cites{corti/weighted,abe/equivariant,azam/equivariant} a weighted Grassmannian was defined in terms of coweights of \(\GL(V) \times \mathbb{G}_m\) (see Remark~\ref{remark:GL-construction}).

In this paper we propose an explicit construction of a weighted generalised Grassmannian. Namely, for a given fundamental representation \(\psi \colon G \hookrightarrow \GL(W)\) of any simple algebraic group \(G\) we provide an explicit method of describing the coweight lattice of the group \(G(\psi) = \psi(G) \cdot Z(\GL(W))\) and computing its integral basis. It has to be noted that in general it involves cumbersome calculations for any fundamental representation on a case-by-case basis (see Example~\ref{example:weighted-quadrics}). In this paper we also provide a complete description of \(\mathbb{Z}_{+}\)-gradings for \emph{weighted Grassmannians} (i.e., for series A). It coincides with the natural description of the \(\mathbb{Z}_{+}\)-grading of a weighted projective space \(\mathbb{P}(a_0, \ldots, a_n)\), but in a general case there are certain differences.

Namely, in the case of a weighted projective space we have an obvious permutation action on its \(\mathbb{Z}\)-grading, preserving the variety up to isomorphism. It can be checked that this action can be identified with the action of the Weyl group \(\mathcal{W} \simeq S_{n + 1}\). For an arbitrary weighted generalised Grassmannian we also have the Weyl action on the \(\mathbb{Z}\)-grading, and it also preserves the variety up to isomorphism (see Section~\ref{section:construction}). We show that for a weighted Grassmannian the integral representation of \(\mathcal{W}\) is not a permutation representation, so there is no ``canonical'' description of \(\mathbb{Z}\)-gradings. Yet, we provide a ``sufficiently nice'' integral basis for \(G(\psi) = \psi(G) \cdot Z(\GL(W))\).

It is easy to see that the coweight lattice of the group \(G(\psi) = \psi(G) \cdot Z(\GL(W))\) contains all possible coweights of \(\GL(W)\) acting on the affine cone \(\widehat{Y}\) of the Grassmannian \(Y\). Actually, recall that \(\Aut^0(Y) \simeq \PGL(V)\) (for example, see~\cite{demazure/automorphisms}). Let us denote by \(T_{G(\psi)} = \psi(T_G) \cdot Z(\GL(W))\) the maximal torus of \(G(\psi)\). Note that Schur's lemma implies that \(\psi(Z(G)) = \psi(G) \cap Z(\GL(W))\), so we can identify \(T_{G(\psi)} / Z(\GL(W)) \simeq T_G / Z (G)\)  with a maximal torus of \(\PGL(V) \simeq \Aut^0(Y)\).

Let us formulate the main results of this paper.

\begin{theorem}[see Corollary~\ref{corollary:SL-grading}]\label{theorem:SL-grading}
  Let \(\widehat{Y} // \mathbb{G}_m\) be a weighted Grassmannian. Then there exists a tuple \((a_0, \ldots, a_n) \in \mathbb{Z}^{n + 1}\) such that this variety is isomorphic to a subvariety \(Y = \Gr_k(a_0, \ldots, a_n) \subset \mathrm{w} \mathbb{P}\), where we define \(Y\) and \(\mathrm{w} \mathbb{P}\) as follows.

  Let \(\psi \colon G \hookrightarrow \GL(W)\), \(W = \Lambda^k(V)\), be a fundamental representation of \(G = \SL(V)\). We choose the Pl\"{u}cker coordinates on \(V\) and \(W\):
  \[
    \Sym(V^{\vee}) \simeq \Bbbk[T_0, \ldots, T_n], \quad
    \Sym(W^{\vee}) \simeq \Bbbk[\{T_{i_1 < \cdots < i_k}\}].
  \]
  Let \(\mathrm{w} \mathbb{P}\) be the weighted projective space defined by the following \(\mathbb{Z}_{+}\)-grading on \(W\):
  \[
    \deg(T_{i_1 < \cdots < i_k}) =  -\sum_{l = 0}^{k - 2} a_l + \sum_{j = 1}^k a_{i_j} > 0.
  \]
  Then we denote by \(Y \subset \mathrm{w} \mathbb{P}\) its closed subvariety defined by usual Pl\"{u}cker relations.
\end{theorem}

\begin{remark}\label{remark:GL-construction}
  In~\cites{corti/weighted,abe/equivariant,azam/equivariant} a weighted Grassmannian is defined in terms of \(\GL(V) \times \mathbb{G}_m\). The corresponding parametrisation of \(\mathbb{Z}\)-gradings has the form
  \[
    w_0, \ldots, w_n, u \in \mathbb{Z}, \quad
    \deg(T_{i_1 < \cdots < i_k}) = \sum_{j = 1}^k w_{i_j} + u.
  \]
  On the one hand, this approach is essentially equivalent to ours:
  \[
    a_i = \left ( \sum_{l = 0}^{k - 2} w_l \right) + u + w_i, \quad
    \deg(T_{i_1 < \cdots < i_k}) = -\sum_{l = 0}^{k - 2} a_l + \sum_{j = 1}^k a_{i_j} =
    \sum_{j = 1}^k w_{i_j} + u.
  \]
  Note also that we can obtain \((a_0, \ldots, a_n) = (0, \ldots, 0, 1, 0, \ldots, 0)\), where \(1\) is in \(i\)-th place, as \(w_i = 1\), \(w_j = u = 0\) if \(i > k - 2\), and as \(w_i = 1\), \(u = -1\), \(w_j = 0\) otherwise. On the other hand, this parametrisation is clearly not effective, because any choice of parameters \((w_0, \ldots, w_n; u)\) is defined only up to adding
  \[
    (w_0, \ldots, w_n; u) = \left ( -\frac{1}{k}, \ldots, -\frac{1}{k}, 1 \right).
  \]
  Therefore to eliminate \(u \in \mathbb{Z}\), we have to allow \(w_i\) to be rational, which leads to combinatorial complications (this phenomenon was first observed in~\cite{corti/weighted}).

  Apart from that, our approach can be straightforwardly generalised to other generalised Grassmannians (see the discussion below), while the described approach is type A specific. This is potentially useful for the classification of \(\mathbb{Q}\)-Fano threefolds (see~\cite{brown/grdb}): the extra \(\mathbb{Z}\)-gradings might yield new families of \(\mathbb{Q}\)-Fano threefolds.
\end{remark}

Moreover, we also explicitly describe the dualising sheaf and Hilbert series of a weighted Grassmannian in terms of the introduced parametrisation of \(\mathbb{Z}\)-gradings. Here we follow the notations of Theorem~\ref{theorem:SL-grading}.

\begin{theorem}[see Proposition~\ref{proposition:dualising-sheaf}]\label{theorem:dualising-sheaf}
  Let \(Y = \Gr_k(a_0, \ldots, a_n) \subset \mathrm{w} \mathbb{P}\) be a well-formed weighted Grassmannian. The dualising sheaves on \(Y\) and \(\mathrm{w} \mathbb{P}\) are isomorphic to
  \begin{gather*}
    \omega_Y \simeq \mathcal{O}_Y
    \left (
      - k
      \left (
        \sum_{i = 0}^n a_i
      \right ) +
      (n + 1)
      \left (
        \sum_{i = 0}^{k - 2} a_i
      \right )
    \right), \\
    \omega_{\mathrm{w} \mathbb{P}} \simeq \mathcal{O}_{\mathrm{w} \mathbb{P}}
    \left (
      - \binom{n}{k - 1}
      \left (
        \sum_{i = 0}^n  a_i
      \right ) +
      \binom{n + 1}{k}
      \left (
        \sum_{i = 0}^{k - 2} a_i
      \right )
    \right ).
  \end{gather*}
\end{theorem}

\begin{remark}
  Theorem~\ref{theorem:dualising-sheaf} and the preceding discussion of the Weyl action seem to be in contradiction. On the one hand, the Weyl action of \(\mathcal{W} \simeq S_{n + 1}\) preserves a weighted Grassmannian up to isomorphism. On the other hand, the formulae of Theorem~\ref{theorem:dualising-sheaf} are clearly asymmetrical: we have ``distinguished'' parameters \((a_0, \ldots, a_{k - 2})\) among all parameters \((a_0, \ldots, a_n)\). As we have already mentioned, in general the coweight lattice as an integral representation of the Weyl group does not admit a permutation basis. In particular, the integral basis of Theorem~\ref{theorem:SL-grading} is not a permutation basis. Note that permutations of parameters \((a_0, \ldots, a_n)\) which preserve the partition of \(\{0, \ldots, n\}\) into \(\{0, \ldots, k - 2\}\) and its complement also preserve the formulae of Theorem~\ref{theorem:dualising-sheaf}. For an arbitrary permutation its action on parameters \((a_0, \ldots, a_n)\) is not obvious, but can be explicitly computed as described in Remark~\ref{remark:weyl-action}.

  To sum it up, the Weyl action of \(\mathcal{W} \simeq S_{n + 1}\) permutes values of any given \(\mathbb{Z}\)-grading, but \(\mathcal{W}\) usually acts non-obviously on any parametrisation of all \(\mathbb{Z}\)-gradings.
\end{remark}

\begin{example}
  In the setting of Remark~\ref{remark:GL-construction} the formula of Theorem~\ref{theorem:dualising-sheaf} reads as
  \[
    \omega_Y \simeq \mathcal{O}_Y
    \left (
      - k \left ( \sum_{i = 0}^n w_i \right ) - (n + 1) u
    \right).
  \]
  So it generalises the similar formulae in~\cites{corti/weighted,qureshi/flag-I}, as to be expected.
\end{example}

\begin{theorem}[see Proposition~\ref{proposition:hilbert-series}]
  Let \(Y = \Gr_k(a_0, \ldots, a_n)\) be a weighted Grassmannian. Its Hilbert series (see Definition~\ref{definition:hilbert-series}) equals \(\mathcal{H}(Y) = \mathcal{P}(Y) / \mathcal{Q}(Y)\), where
  \begin{gather*}
    \mathcal{P}(Y) = \sum_{I = (i_1 < \cdots < i_k)} \frac{F(I)}{1 - t^{a(I)}}, \quad
    \mathcal{Q}(Y) = \prod_{\substack{i,j = 0, \ldots, n \\ i < j}} (1 - t^{a_j - a_i}); \\
    a(I) = -\sum_{l = 0}^{k - 2} a_l + \sum_{j = 1}^k a_{i_j} > 0, \quad
    F(I) =
    \left (
      \sum_{\sigma \in S_I} (-1)^{\sigma} t^{f(\sigma)}
    \right ); \\
    S_I = \{\sigma \in S_{n + 1} : \sigma(\{0, \ldots, k - 1\}) = I\}, \quad
    f(\sigma) = \sum_{i = 0}^n (i - \sigma^{-1}(i)) a_i.
  \end{gather*}
\end{theorem}

Our approach mainly deals with the combinatorics of irreducible root systems, so we expect that it can be extended to other weighted generalised Grassmannians (see Example~\ref{example:weighted-quadrics} and Question~\ref{question:grading}). Note that the case of series B should be easy, since \(\SO_{2 n + 1}\) is of adjoint type (see Example~\ref{example:series-B}). Weighted generalised Grassmannians (and, more generally, weighted flag varieties) were studied in~\cite{graham/weighted} (see also~\cites{abe/equivariant,azam/equivariant}) from a uniform Lie-theoretic point of view.

\begin{structure}
  In Section~\ref{section:construction} we provide an explicit construction of a weighted generalised Grassmannian. We describe how one can extract a \(\mathbb{Z}\)-grading from the representation-theoretic data. We also discuss basic geometric properties of weighted generalised Grassmannians.

  In Section~\ref{section:series-A-grading} we apply this approach to weighted Grassmannians to obtain an effective parametrisation of \(\mathbb{Z}\)-gradings. In Section~\ref{section:series-A-invariants} we describe the dualising sheaf and the Hilbert series of a weighted Grassmannian in terms of this parametrisation.
\end{structure}

\begin{conventions}
  We work over an algebraically closed field \(\Bbbk\) of characteristic zero. Throughout the paper \(G\) is a simple linear algebraic group, and \(\psi \colon G \hookrightarrow \GL(W)\) is a fundamental representation. For explicit computations we follow the standard numeration of simple roots from~\cite{bourbaki/lie}.

  A bit on terminology: usual representation-theoretic terminology (namely, ``weights'' and ``coweights'') causes confusion with weights of a weighted projective space. In this paper we refer to weights of a weighted projective space as its \(\mathbb{Z}\)-grading. We hope that it would not confuse the reader.
  
  We refer the reader to Sections~\ref{section:construction} and~\ref{section:series-A-grading} for various notations for weighted generalised Grassmannians, and for notations relevant to the case \(G = \SL(V)\), correspondingly.
\end{conventions}

\begin{acknowledgements}
  The article was prepared within the framework of the project ``International academic cooperation'' HSE University. The author is grateful to V.~Przyjalkowski for valuable suggestions, and to W.~Graham and S.~Larson for helpful comments. We also want to thank the referee for useful remarks.
\end{acknowledgements}

\section{Weighted generalised Grassmannians}\label{section:construction}

In this section we describe our construction of a weighted generalised Grassmannian, and outline a way for explicit description of its possible \(\mathbb{Z}\)-gradings.

\subsection{Preliminaries}

Let us start with introducing some notations. We refer the reader to~\cites{humphreys/lie,malle/groups} for the representation-theoretic background.

\begin{notation}
  Let \(G\) be a simple linear algebraic group. We always choose a maximal torus \(T_G \subset G\). We denote by \(X(T_G)\) and \(Y(T_G)\) the weight and coweight lattices of \(T_G\), correspondingly. We also denote by \(\Phi \subset X(T_G)\) and \(\Phi^{\vee} \subset Y(T_G)\) the root system of \(G\) and its dual.

  We fix a basis \(\Delta = \{\alpha_1, \ldots, \alpha_n\} \subset \Phi\) of the root system \(\Phi\), and denote by \(\{\alpha_1^{\vee}, \ldots, \alpha_n^{\vee}\} \subset \Phi^{\vee}\) the associated simple coroots.
\end{notation}

\begin{definition}
  Let \(P \subset G\) be a maximal parabolic subgroup. We refer to the homogeneous space \(Y = G / P\) as a \emph{generalised Grassmannian}.
\end{definition}

\begin{remark}
  A generalised Grassmannian \(Y\) admits a minimal \(G\)-equivariant embedding \(Y \hookrightarrow \mathbb{P}\), where \(\psi \colon G \hookrightarrow \GL(W)\) is a fundamental representation of \(G\) associated with the choice of a maximal parabolic subgroup \(P\).
\end{remark}

\begin{notation}
  For a fundamental representation \(\psi \colon G \hookrightarrow \GL(W)\) we choose a maximal torus \(T_W \subset \GL(W)\). We denote by \(G(\psi)\) the product \(\psi(G) \cdot \mathbb{G}_m\) in \(\GL(W)\), where \(\mathbb{G}_m = Z(\GL(W))\). Note that \(G(\psi)\) is equipped with a maximal torus of the form \(T_{G(\psi)} = \psi(T_G) \cdot \mathbb{G}_m \subset T_W\). We denote by \(X(G(\psi))\) and \(Y(G(\psi))\) the weight and coweight lattices of the group \(G(\psi)\). We also denote by \(Y_+(G(\psi)) \subset Y(G(\psi))\) the subsemigroup of coweights endowing the vector space \(W\) with a \(\mathbb{Z}_{+}\)-grading. At last, we denote by \(\tau^{\vee}_W \colon \mathbb{G}_m \hookrightarrow \GL(W)\), \(\tau^{\vee}_W(c) = c \cdot \Id_W\), the central coweight of \(\GL(W)\). 
\end{notation}

\subsection{The construction}

Now we can define weighted generalised Grassmannians.

\begin{definition}[cf.~{\cites{corti/weighted,qureshi/flag-I}}]
  Let \(\psi \colon G \hookrightarrow \GL(W)\) be a fundamental representation. Then \(G(\psi) = \psi(G) \cdot \mathbb{G}_m\) acts on the affine cone \(\widehat{Y} \subset \mathbb{A}(W)\) of the generalised Grassmannian, where \(\mathbb{A}(W)\) is the associated affine space of \(W\). The \emph{weighted generalised Grassmannian} is a GIT quotient \(\widehat{Y} // \mathbb{G}_m\) by a coweight \(\gamma^{\vee} \in Y_+(G(\psi))\) endowing the vector space \(W\) with a \(\mathbb{Z}_{+}\)-grading.
\end{definition}

\begin{example}\label{example:WPS-lattice}
  Let \(G = \SL(V)\), and \(\psi \colon G \hookrightarrow \GL(V)\) be a standard representation, so \(G(\psi) = \GL(V)\). After a choice of a basis on \(V\) the coweights in \(Y_+(G(\psi))\) can be identified with positive integer vectors \((a_0, \ldots, a_n) \in \mathbb{Z}^{n + 1}\). In other words, we obtain weighted projective spaces \(\mathbb{P}(a_0, \ldots, a_n)\) for any possible \(\mathbb{Z}_{+}\)-grading on \(V\).
\end{example}

\begin{remark}\label{remark:quotient}
  By construction any weighted generalised Grassmannian \(Y \subset \mathrm{w} \mathbb{P}\) is equipped with the embedding into a weighted projective space, and it is defined there by the same (generalised) Pl\"{u}cker relations. Let us also recall that there is a natural finite morphism \(\mathbb{P} \twoheadrightarrow \mathrm{w} \mathbb{P}\) realising \(\mathrm{w} \mathbb{P}\) as a quotient of \(\mathbb{P}\) by a finite abelian group acting on \(\mathbb{P}\) by roots of unity. This realises \(Y\) as a quotient of a usual generalised Grassmannian \(G / P \hookrightarrow \mathbb{P}\) by the same finite abelian group.

  In particular, this implies that \(Y\) is covered by quotients of \(\mathbb{A}^{\dim(Y)}\) by finite cyclic groups, which are images of big Schubert cells in \(G / P\) (see~\cite[\nopp 3.3]{lakshmibai/monomial}).
\end{remark}

\begin{proposition}
  A weighted generalised Grassmannian \(Y\) is a Mori dream space with \(\Cl(Y) \simeq \mathbb{Z}\).
\end{proposition}

\begin{proof}
  By construction \(Y \simeq \Proj(R)\), where \(R\) is the homogeneous coordinate ring of a generalised Grassmannian \(\widetilde{Y}\) (equipped with the corresponding \(\mathbb{Z}_{+}\)-grading). It is well-known that \(\widetilde{Y}\) is a Mori dream space with Cox ring \(R\) (with its standard \(\mathbb{Z}\)-grading). Vice versa, for a given so-called \emph{bunched ring} \(R\) (informally speaking, a graded ring with additional combinatorial data) one can recover a Mori dream space with Cox ring \(R\) (see~\cite[\nopp 3.2]{arzhantsev/cox}). It is well-known that \(\widetilde{Y}\) arises in this way (see~\cite[\nopp 3.2.3]{arzhantsev/cox}). Following the exposition in~\cite[\nopp 3.2.3]{arzhantsev/cox}, it may be checked that for a non-standard \(\mathbb{Z}_{+}\)-grading \(R\) is still a bunched ring. Then \(Y\) is a Mori dream space with Cox ring \(R\) and \(\Cl(Y) \simeq \mathbb{Z}\) by~\cite[Theorem~3.2.14]{arzhantsev/cox}.
\end{proof}

\subsection{Computation of the lattice \texorpdfstring{\(Y(G(\psi))\)}{Y(G(psi))}}\label{subsection:lattice-computations}

\begin{remark}\label{remark:rational-splitting}
  In our notation we have a natural chain of inclusions of coweight lattices \(Y(G) \hookrightarrow Y(G(\psi)) \hookrightarrow Y(\GL(W))\) and the following exact sequence:
  \[
    0 \rightarrow Y(G) \rightarrow Y(G(\psi)) \xrightarrow{\det} \Hom(\mathbb{G}_m, \mathbb{G}_m)
    \rightarrow \mathbb{Z} / \dim(W) \mathbb{Z} \rightarrow 0.
  \]
  As a consequence, the \emph{rational} coweight lattice \(Y(G(\psi))_{\mathbb{Q}} = Y(G(\psi)) \otimes_{\mathbb{Z}} \mathbb{Q}\) admits a splitting \(Y(G(\psi))_{\mathbb{Q}} \simeq Y(G)_{\mathbb{Q}} \oplus \Hom(\mathbb{G}_m, \mathbb{G}_m)_{\mathbb{Q}}\) via the coweight \(\dim(W)^{-1} \cdot \tau^{\vee}_W\).
\end{remark}

\begin{definition}
  Put \(\Omega = \Hom(\mathbb{Z} \Phi^{\vee}, \mathbb{Z})\). We have the inclusions \(\mathbb{Z} \Phi \subset X(G) \subset \Omega\), where \(X(G) \simeq \Hom(Y(G), \mathbb{Z}) \hookrightarrow \Hom(\mathbb{Z} \Phi^{\vee}, \mathbb{Z}) = \Omega\). The \emph{fundamental group} of \(G\) is defined as \(\Lambda(G) = \Omega / X(G)\). If \(X(G) = \Omega\), then the group \(G\) is \emph{simply connected}, and if \(X(G) = \mathbb{Z} \Phi\), then \(G\) is of \emph{adjoint type}.
\end{definition}

\begin{definition}\label{definition:fundamental-coweights}
  In the introduced notation the integral weights \(\omega_i \in X(G)\) s.t. \(\langle \omega_i, \alpha_j^{\vee} \rangle = \delta_{ij}\) for any \(i, j\) are called \emph{fundamental weights} of the group \(G\). Similarly, \emph{rational} coweights \(\omega_j^{\vee} \in Y(G)_{\mathbb{Q}}\) s.t. \(\langle\alpha_i, \omega_j^{\vee}\rangle = \delta_{ij}\) are called \emph{fundamental coweights} of \(G\). Note that they are integral if and only if \(G\) is of adjoint type.
\end{definition}

\begin{remark}\label{remark:coweight-lattice-description}
  Note that the rational coweight lattice \(Y(G(\psi))_{\mathbb{Q}}\) from Remark~\ref{remark:rational-splitting} is  generated by fundamental coweights \(\{\omega_i^{\vee}\}\) and the coweight \(\dim(W)^{-1} \cdot \tau^{\vee}_W\). More precisely, the integral coweight lattice \(Y(G(\psi))\) admits a natural embedding to the integral lattice \(L(\psi) = \mathbb{Z} \langle\{\omega_1^{\vee}, \ldots, \omega_n^{\vee}, \dim(W)^{-1} \cdot \tau^{\vee}_W\}\rangle \subset Y(\GL(W))_{\mathbb{Q}}\).

  In other words, we can compute our coweight lattice \(Y(G(\psi))\) as the intersection of two lattices \(L(\psi) \cap Y(\GL(W))\) inside \(Y(\GL(W))_{\mathbb{Q}}\).
\end{remark}

\begin{example}\label{example:series-B}
  \(G = \SO_{2 n + 1}\) is of adjoint type, so \(Y(G(\psi)) \simeq Y(G) \times \End(\mathbb{G}_m)\).
\end{example}

\begin{remark}\label{remark:cartan-matrix}
  The Cartan matrix \(C\) of the root system \(\Phi\) (with respect to its basis~\(\Delta\)) provides a change of coordinates from the fundamental weights \(\{\omega_i\}\) to simple roots \(\{\alpha_i\}\). Dually, the transposed inverse Cartan matrix \(C^{-T}\) provides a change of coordinates from the simple coroots \(\{\alpha_i^{\vee}\}\) to fundamental coweights \(\{\omega_i^{\vee}\}\).
\end{remark}

\begin{computation}\label{computation:lattice-calculations}
  Remarks~\ref{remark:coweight-lattice-description} and~\ref{remark:cartan-matrix} provide us with a natural way to explicitly compute the coweight lattice \(Y(G(\psi))\). Namely, let us identify the coweight lattice \(Y(\GL(W))\) with the standard lattice \(\mathbb{Z}^m\). Recall that the images of simple coroots \(\psi(\alpha^{\vee}_i) \in Y(\GL(W))\) are well-described (for example, see~\cite[\nopp VIII.13]{bourbaki/lie}). Then we can compute the lattice \(Y(G(\psi))\) as the intersection \(L(\psi) \cap \mathbb{Z}^m\) inside \(\mathbb{Q}^m\).

  Note that this computation actually depends only on the root system \(\Phi\). More precisely, we are actually interested only with maximal tori of the linear groups \(G\) and \(G(\psi)\), so in order to perform explicit computations we can work with the Cartan subalgebra \(\mathfrak{h} \subset \Lie(G)\) of the associated Lie algebra.
\end{computation}

\begin{example}\label{example:weighted-quadrics}
  Despite that Computation~\ref{computation:lattice-calculations} allows us to reduce the description of possible \(\mathbb{Z}\)-gradings to direct calculations, we still have to construct an integral basis of \(L(\psi) \cap \mathbb{Z}^m\). In each particular case this can be done by applying Zassenhaus algorithm (see~\cite[207--210]{fischer/lineare}), yet this leads to cumbersome calculations even in simple situations. For example, let \(\psi \colon G \hookrightarrow \GL(V)\) be the standard representation of \(G = \SO(10)\). The associated generalised Grassmannian \(\widetilde{Y}\) is given by
  \[
    Q = X_0 X_5 + X_1 X_6 + X_2 X_7 + X_3 X_8 + X_4 X_9 = 0.
  \]
  By applying Computation~\ref{computation:lattice-calculations} it is possible to check that the coweight lattice \(Y(G(\psi))\) is generated by the following elements (written in the standard basis):  
  \begin{gather*}
    \omega_1^{\vee} = (1, 0, 0, 0, 0, -1, 0, 0, 0, 0), \quad
    \omega_2^{\vee} - \omega_1^{\vee} = (0, 1, 0, 0, 0, 0, -1, 0, 0, 0), \\
    \omega_3^{\vee} - \omega_2^{\vee} = (0, 0, 1, 0, 0, 0, 0, -1, 0, 0), \quad
    \omega_4^{\vee} + \omega_5^{\vee} - \omega_3^{\vee} = (0, 0, 0, 1, 0, 0, 0, 0, -1, 0), \\
    \omega_5^{\vee} - \omega_4^{\vee} = (0, 0, 0, 0, 1, 0, 0, 0, 0, -1), \quad
    \dfrac{1}{2} \tau^{\vee}_V - \omega_5^{\vee} = (0, 0, 0, 0, 0, 1, 1, 1, 1, 1),
  \end{gather*}
  where we have \(\langle \alpha_i,  \omega_j^{\vee} \rangle = \delta_{ij}\). In other words, the \(\mathbb{Z}\)-gradings are parametrised by
  \[
    \deg(X_i) = a_i, \quad \deg(X_{5 + i}) = - a_i + a_5, \quad i = 0, \ldots, 4; \quad \deg(Q) = a_5.
  \]
  Note that if \(a_5\) is odd, then such a \(\mathbb{Z}\)-grading is not an integral coweight of \(G \times \mathbb{G}_m\).

  There also exist other ways for computing an integral basis of the coweight lattice, at least for non-exceptional simple groups (see Proposition~\ref{proposition:SL-grading} and Question~\ref{question:grading}).
\end{example}

\subsection{A canonical integral basis of \texorpdfstring{\(Y(G(\psi))\)}{Y(G(psi))}, and absence thereof}

We also have to provide a ``canonical'' integral basis for the coweight lattice \(Y(G(\psi))\). It is easy to see from the discussion in Subsection~\ref{subsection:lattice-computations} that the lattice \(Y(G(\psi))\) is equipped with the action of the Weyl group \(\mathcal{W}\) of the root system \(\Phi\).
One can see (see the proof of~\cite[Lemma~3.3]{corti/weighted}) that this action preserves a weighted generalised Grassmannian up to isomorphism.
In other words, in its construction we can always replace a coweight \(\gamma^{\vee}\) with a coweight of the form \(w \cdot \gamma^{\vee}\) for any \(w \in \mathcal{W}\).

In Section~\ref{section:series-A-grading} we will see that in the case of Example~\ref{example:WPS-lattice} (i.e., for weighted projective spaces) this action can be identified with a natural permutation action on the \(\mathbb{Z}\)-grading of a weighted projective space. So an obvious thought is to consider the lattice \(Y(G(\psi))\) as an integral representation of the Weyl group, and try to find a permutation basis. Unfortunately, this is not the case for arbitrary weighted Grassmannians, see Remark~\ref{remark:weyl-action}. Note that for series B the structure of an integral representation is quite simple (see Example~\ref{example:series-B}). In general we can try to find a ``sufficiently nice'' integral basis, which is done in Section~\ref{section:series-A-grading} for weighted Grassmannians. It would be interesting to complete this description for other irreducible root systems.

\subsection{Well-formedness}

Similarly to the case of weighted projective spaces, we often have to restrict ourselves to \emph{well-formed} weighted Grassmannians (see~\cite{przyjalkowski/weighted} for relevant definitions and pathological examples related to weighted projective spaces).

\begin{definition}
  A weighted generalised Grassmannian \(Y \subset \mathrm{w} \mathbb{P}\) is \emph{well-formed} if \(\mathrm{w} \mathbb{P}\) is a well-formed weighted projective space, and \(\codim_Y(Y \cap \Sing(\mathrm{w} \mathbb{P})) \geqslant 2\).
\end{definition}

\begin{remark}
  It is expected (see~\cite[Problem~3.4]{corti/weighted}) that \(Y \subset \mathrm{w} \mathbb{P}\) is well-formed if and only if \(\mathrm{w} \mathbb{P}\) is a well-formed weighted projective space.
\end{remark}

\begin{lemma}\label{lemma:singular-locus}
  If \(Y \subset \mathrm{w} \mathbb{P}\) is a well-formed weighted generalised Grassmannian, then we have \(\Sing(Y) = Y \cap \Sing(\mathrm{w} \mathbb{P})\).
\end{lemma}

\begin{proof}
  We can apply the proof of~\cite[Proposition~8]{dimca/singularities} by~\cite[Remark~2.5]{przyjalkowski/on-automorphisms}. Note that the subvariety \(Y \subset \mathrm{w} \mathbb{P}\) is always quasi-smooth, because any generalised Grassmannian is smooth in the usual sense.
\end{proof}

\subsection{Serre twisting sheaf and Hilbert series}

\begin{definition}[see~{\cite[Definition~2.1]{sano/hypersurfaces}}]
  Let \(Y \subset \mathrm{w} \mathbb{P}\) be a weighted generalised Grassmannian (not necessary well-formed). We denote their coordinate rings by \(\mathcal{R}(\mathrm{w} \mathbb{P})\) and \(\mathcal{R}(Y) = \mathcal{R}(\mathrm{w} \mathbb{P}) / I(Y)\), correspondingly, where \(I(Y)\) is the ideal of (generalised) Pl\"{u}cker relations. We define the sheaf \(\mathcal{O}_Y(k)\) as the coherent sheaf on \(Y\) associated to the graded \(\mathcal{R}(Y)\)-module \(R(k)\) whose degree \(i\) part is \(R(k)_i = \mathcal{R}(Y)_{k + i}\).
\end{definition}

\begin{lemma}[see~{\cite[Proposition~2.3(1)]{sano/hypersurfaces}}]\label{lemma:hilbert-series}
  Let \(Y \subset \mathrm{w} \mathbb{P}\) be a weighted generalised Grassmannian. Then we have \(H^0(Y, \mathcal{O}_Y(k)) \simeq \mathcal{R}(Y)_k\) for any \(k \geqslant 0\).
\end{lemma}

\begin{remark}
  If a weighted generalised Grassmannian \(Y \subset \mathrm{w} \mathbb{P}\) is not well-formed, the sheaf \(\mathcal{O}_Y(k)\) may display pathological behaviour (cf.~\cite[Remark~2.5]{sano/hypersurfaces}).
\end{remark}

\begin{definition}\label{definition:hilbert-series}
  Let \(Y \subset \mathrm{w} \mathbb{P}\) be a weighted generalised Grassmannian (not necessary well-formed). We define its \emph{Hilbert series} as a formal series \(\mathcal{H}(Y) = \sum_{k = 0}^{\infty} d_k t^k\), where we put \(d_k = \dim(H^0(Y, \mathcal{O}_Y(k))) = \dim(\mathcal{R}(Y)_k)\) by virtue of Lemma~\ref{lemma:hilbert-series}.
\end{definition}

\begin{lemma}\label{lemma:serre-sheaf-extension}
  Let \(Y \subset \mathrm{w} \mathbb{P}\) be a well-formed weighted generalised Grassmannian. Put \(Y^{\circ} = Y \setminus \Sing(\mathrm{w} \mathbb{P})\), and let \(i \colon Y \hookrightarrow \mathrm{w} \mathbb{P}\) and \(j \colon Y^{\circ} \hookrightarrow Y\) be natural embeddings. Then we have the following isomorphisms of sheaves on \(Y\) for any \(k\):
  \[
    \mathcal{O}_Y(k) \simeq j_* j^* \mathcal{O}_Y(k) \simeq j_* j^* i^* \mathcal{O}_{\mathrm{w} \mathbb{P}}(k).
  \]
\end{lemma}

\begin{proof}
  We can repeat the proof of~\cite[Proposition~2.3]{sano/hypersurfaces}. Actually, recall that a Grassmannian in its Pl\"{u}cker embedding is always \emph{arithmetically Cohen--Macaulay}, i.e., its affine cone is Cohen--Macaulay at its every point, including the vertex (for example, see~\cite[Theorem~4.6.0.2]{lakshmibai/monomial}). If we denote by \(\mathfrak{m} \subset R\) the irrelevant ideal of the homogeneous coordinate ring \(R\) of a Grassmannian, then the assumption \(\depth_{\mathfrak{m}}(R) \geqslant 2\) of~\cite[Proposition~2.3(ii)]{sano/hypersurfaces} is equivalent to \(\height(\mathfrak{m}) \geqslant 2\) (see~\cite[Proposition~1.5.15, Corollary~2.1.4]{bruns/cohenmacaulay}). Finally, the proof of~\cite[Proposition~2.3(iii)]{sano/hypersurfaces} uses only that the affine cone of \(Y \subset \mathrm{w} \mathbb{P}\), i.e., \(\Spec(R)\), is Cohen--Macaulay.
\end{proof}

\section{Integral gradings for weighted Grassmannians}\label{section:series-A-grading}

\subsection{Results and examples}

We put \(G = \SL(V)\) and \(n + 1 = \dim(V) = \rk(G) + 1\). In this section we construct an integral basis for the coweight lattice of the group \(G(\psi) = \psi(G) \cdot Z(\GL(W))\), where \(\psi \colon G \hookrightarrow \GL(W)\) is a fundamental representation.

\begin{proposition}[see~{\cite[\nopp VIII.13.1]{bourbaki/lie}}]
  A fundamental representation of \(\SL(V)\) is isomorphic to an exterior power of the standard representation \(\SL(V) \hookrightarrow \GL(V)\).
\end{proposition}

\begin{notation}
  Let \(\{\omega_i\}\) and \(\{\omega_i^{\vee}\}\) be fundamental weights and coweights of \(G\), correspondingly (see Definition~\ref{definition:fundamental-coweights}). For any \(i = 0, \ldots, n\) we put \(e_i = -\omega_{i - 1} + \omega_i\) and \(e^{\vee}_i = -\omega^{\vee}_{i - 1} + \omega^{\vee}_i\), where \(\omega_0 = \omega_{n + 1} = 0\) and \(\omega^{\vee}_{0} = \omega^{\vee}_{n + 1} = 0\).
\end{notation}

\begin{remark}\label{remark:SL-coordinates}
  Let us choose coordinates \(\{T_i : i = 0, \ldots, n\}\) on the vector space \(V\). Then we can identify \(\{T_i\}\) with the orbit \(\mathcal{W} \cdot \omega_1 = \{e_i\}\) of the Weyl group of the fundamental weight \(\omega_1\) (the highest weight of the standard representation \(V\)). Direct calculations show that \(\psi(e^{\vee}_i) + \dim(V)^{-1} \cdot \tau^{\vee}_V\) can be identified with the \(i\)-th standard vector in \(V\). So \(\{\psi(e^{\vee}_i) + \dim(V)^{-1} \cdot \tau^{\vee}_V\}\) form dual coordinates on \(V^{\vee}\).

  In the same way if \(W = \Lambda^k(V)\) is the representation of the highest weight \(\omega_k\), then the coordinates \(\{T_i : i = 0, \ldots, n\}\) on \(V\) induce the coordinates \(\{T_{i_1 < \cdots < i_k}\}\) on \(W\), and we can identify them with the orbit \(\mathcal{W} \cdot \omega_k\) of the Weyl group (cf. Lemma~\ref{lemma:weyl-orbit}).
\end{remark}

In this notation we can recover the \(\mathbb{Z}_{+}\)-grading of a weighted projective space.

\begin{example}\label{example:WPS}
  A weighted projective space \(\mathbb{P}(a_0, \ldots, a_n) = \Proj(\Bbbk[T_0, \ldots, T_n])\) is defined by a coweight of the form \(\gamma^{\vee} = \psi(\mu^{\vee}) + m \cdot \tau^{\vee}_W\), where we put
\[
  \mu^{\vee} = \sum_{i = 1}^n (a_{i - 1} - a_i) \omega_i^{\vee}, \quad
  m = \frac{\sum_{i = 0}^n a_i}{n + 1}, \quad 
  \deg(T_i) = \langle e_i, \mu^{\vee} \rangle + m = a_i.
\]
\end{example}

The goal of this section is to prove the following generalisation of Example~\ref{example:WPS}.

\begin{proposition}\label{proposition:SL-grading}
  Let \(\psi \colon G \hookrightarrow \GL(W)\), \(W = \Lambda^k(V)\), be a fundamental representation. Then the coweight lattice of the group \(G(\psi) = \psi(G) \cdot Z(\GL(W))\) admits the following integral basis:
  \[
    \gamma^{\vee}_i =
    \begin{cases}
      (k / (n + 1) - 1) \cdot \tau^{\vee}_W + \psi(e_i^{\vee})  & \text{ for } 0 \leqslant i \leqslant k - 2; \\
      k / (n + 1) \cdot \tau^{\vee}_W + \psi(e_i^{\vee})  & \text{ for } k - 1 \leqslant i \leqslant n. \\
    \end{cases}
  \]
\end{proposition}

We postpone the proof of Proposition~\ref{proposition:SL-grading} until the end of this section. Let us first discuss its immediate consequences.

\begin{corollary}\label{corollary:SL-grading}
  Let \(\psi \colon G \hookrightarrow \GL(W)\), \(W = \Lambda^k(V)\), be a fundamental representation. We fix a basis \(V = \langle e_0, \ldots, e_n \rangle \), so we have the following coordinates on \(V\) and \(W\):
  \[
    \Sym(V^{\vee}) \simeq \Bbbk[T_0, \ldots, T_n], \quad
    \Sym(W^{\vee}) \simeq \Bbbk[\{T_{i_1 < \cdots < i_k}\}],
  \]
  An arbitrary coweight \(\sum_{i = 0}^n a_i \gamma^{\vee}_i\) of \(G(\psi)\) induces on \(W\) the following \(\mathbb{Z}\)-grading:
  \[
    \deg(T_{i_1 < \cdots < i_k}) =  -\sum_{l = 0}^{k - 2} a_l + \sum_{j = 1}^k a_{i_j}.
  \]
\end{corollary}

\begin{proof}
  Let us rewrite our coweight \(\sum_{i = 0}^n a_i \gamma^{\vee}_i\) from Proposition~\ref{proposition:SL-grading} as follows:
  \[
    \sum_{i = 0}^n a_i \gamma^{\vee}_i = \sum_{i = 1}^n (a_{i - 1} - a_i) \omega_i^{\vee} + m \cdot \tau^{\vee}_W, \quad
    m = \left ( \frac{k}{n + 1} \cdot \sum_{i = 0}^{n} a_i - \sum_{i = 0}^{k - 2} a_i \right ).
  \]
  Then Remark~\ref{remark:SL-coordinates} implies that
  \[
    \deg(T_{i_1 < \cdots < i_k}) = \left \langle \sum_{j = 1}^k e_{i_j},
    \sum_{i = 1}^n (a_{i - 1} - a_i) \omega_i^{\vee} \right \rangle + m.
  \]
  Using Lemma~\ref{lemma:SL-lattice description} below, it is not hard to compute this pairing:
  \[
    \left \langle \sum_{j = 1}^k e_{i_j},
      \sum_{i = 1}^n (a_{i - 1} - a_i) \omega_i^{\vee} \right \rangle =
    -\frac{k}{n + 1} \left ( \sum_{i = 0}^n a_i \right ) + \sum_{j = 1}^k a_{i_j}.
  \]
  Actually, this can be checked for each \((a_0, \ldots, a_n) = (0, \ldots, 0, 1, 0, \ldots, 0)\).
\end{proof}

\begin{remark}
  Note that \(\deg(T_{0, \ldots, k - 2, l}) = a_l\) for all \(l = k - 1, \ldots, n\). In particular, for \(k = 1\) this means that \(\deg(T_i) = a_i\) for all \(i = 0, \ldots, n\), as to be expected.
\end{remark}

\begin{example}
  For \(k = 2\), \(n = 4\), the Pl\"{u}cker relations are given by
  \begin{gather*}
    T_{0, 1} T_{2, 3} + T_{1, 2} T_{0, 3} - T_{0, 2} T_{1, 3} = 0, \quad
    T_{0, 1} T_{2, 4} + T_{1, 2} T_{0, 4} - T_{0, 2} T_{1, 4} = 0, \\
    T_{0, 1} T_{3, 4} + T_{1, 3} T_{0, 4} - T_{0, 3} T_{1, 4} = 0, \quad
    T_{0, 2} T_{3, 4} + T_{2, 3} T_{0, 4} - T_{0, 3} T_{2, 4} = 0, \\
    T_{1, 2} T_{3, 4} + T_{2, 3} T_{1, 4} - T_{1, 3} T_{2, 4} = 0.
  \end{gather*}
  We can clearly see that those equations are quasi-homogeneous of degrees
  \begin{gather*}
    - a_0 + a_1 + a_2 + a_3, \; -a_0 + a_1 + a_2 + a_4, \;
    - a_0 + a_1 + a_3 + a_4, \\
    - a_0 + a_2 + a_3 + a_4, \; - 2 a_0 + a_1 + a_2 + a_3 + a_4.
  \end{gather*}
\end{example}

\begin{remark}\label{remark:weyl-action}
  The Weyl group \(\mathcal{W}\) acts on the coweight lattice of \(G(\psi)\) by (inverse) permutations of \(\psi(e_i^{\vee})\), and leaves \(\tau^{\vee}_W\) in its place. We also can replace \(\{0, \ldots, k - 2\}\) with any subset of \(\{0, \ldots, n\}\) of the same size, but this will yield a different integral basis. In other words, if we realise the coweight lattice as an integral submodule of
  \[
    Y(G(\psi))_{\mathbb{Q}} \simeq \mathbb{Q} \langle\psi(e_0^{\vee}), \ldots, \psi(e_n^{\vee}), \deg(W)^{-1} \cdot \tau^{\vee}_{W}\rangle
  \]
  generated by columns of
  \begin{gather*}
    M =
    \begin{pmatrix}
      1 & -1 & 0 & \cdots & 0 & 0 & 0 \\ 
      0 & 1 & -1 & \cdots & 0 & 0 & 0 \\
      \cdots & \cdots & \cdots & \cdots & \cdots & \cdots & \cdots \\
      0 & 0 & 0 & \cdots & 1 & -1 & 0 \\
      0 & 0 & 0 & \cdots & 0 & 1 & -1 \\
      \alpha_k & \alpha_k & \alpha_k & \cdots & \beta_k & \beta_k & \beta_k
    \end{pmatrix}, \quad
    \begin{cases}
      \alpha_k = (\frac{k}{n + 1} - 1) \binom{n + 1}{k}, \\
      \beta_k = \frac{k}{n + 1} \binom{n + 1}{k} = \binom{n}{k - 1}, \\
      \det(M) = (n - k + 2) \beta_k + \\
      \quad + (k - 1) \alpha_k = \binom{n + 1}{k},
    \end{cases}
  \end{gather*}
  then \(\mathcal{W}\) acts by (inverse) permutations of its columns while preserving the last row. For example, we send \((1, 0, \ldots, 0, \alpha_k)\) to \((0, \ldots, 0, -1, \alpha_k)\), not \((0, \ldots, 0, -1, \beta_k)\), because \(\mathcal{W}\) leaves the central coweight \(\tau^{\vee}_W\) in its place. Moreover, different choices of the subset in \(\{0, \ldots, n\}\) amount to permutations of elements of the last row.

  It is easy to check that the coweight lattice is a permutation representation of the Weyl group \(\mathcal{W}\) if and only if \(\min(k, n + 1 - k) = 1\) (cf. the proof of Corollary~\ref{corollary:SL-grading}).
\end{remark}

\begin{question}
  In other words, Remark~\ref{remark:weyl-action} implies that the coweight lattice of \(G(\psi)\) provides a non-trivial \(\mathbb{Z}\)-form of the standard permutation representation of \(\mathcal{W} \simeq S_{n + 1}\). How do they (explicitly) look like for other irreducible root systems?
\end{question}

\begin{remark}
  It is not hard to check that for \(\widetilde{k} = n + 1 - k\) the following matrices are column-equivalent over \(\mathbb{Z}\) (cf. the induction basis of the proof of Proposition~\ref{proposition:SL-grading}):
  \vspace{0.6cm}
  \begin{gather*}
    \begin{pNiceMatrix}
      1 & -1 & 0 & \cdots & 0 & 0 & 0 \\ 
      0 & 1 & -1 & \cdots & 0 & 0 & 0 \\
      \cdots & \cdots & \cdots & \cdots & \cdots & \cdots & \cdots \\
      0 & 0 & 0 & \cdots & 1 & -1 & 0 \\
      0 & 0 & 0 & \cdots & 0 & 1 & -1 \\
      \alpha_k & \alpha_k & \alpha_k & \cdots & \beta_k & \beta_k & \beta_k
      \CodeAfter
      \OverBrace{1-1}{6-3}{k - 1}
      \OverBrace{1-5}{6-7}{\widetilde{k} + 1}
    \end{pNiceMatrix}
    \sim
    \begin{pNiceMatrix}
      -1 & 1 & 0 & \cdots & 0 & 0 & 0 \\ 
      0 & -1 & 1 & \cdots & 0 & 0 & 0 \\
      \cdots & \cdots & \cdots & \cdots & \cdots & \cdots & \cdots \\
      0 & 0 & 0 & \cdots & -1 & 1 & 0 \\
      0 & 0 & 0 & \cdots & 0 & -1 & 1 \\
      \beta_{\widetilde{k}} & \beta_{\widetilde{k}} & \beta_{\widetilde{k}} & \cdots &
      \alpha_{\widetilde{k}} & \alpha_{\widetilde{k}} & \alpha_{\widetilde{k}}
      \CodeAfter
      \OverBrace{1-1}{6-3}{\widetilde{k} - 1}
      \OverBrace{1-5}{6-7}{k + 1}
    \end{pNiceMatrix}.
  \end{gather*}
  Note that \(\alpha_{\widetilde{k}} + \beta_k = \alpha_k + \beta_{\widetilde{k}} = 0\). Let us identify vector spaces of the representations:
  \[
    W = \Lambda^k(V) \xrightarrow{\sim} W^{\vee} = \Lambda^{n + 1 - k} (V), \quad
    T_I \mapsto T_{I^{\circ}}, \quad
    I^{\circ} = \{0, \ldots, n\} \setminus I.
  \]
  Arguing as in the proof of Corollary~\ref{corollary:SL-grading}, we see the this new integral basis of the coweight lattice (for some \(k\)) yields the same \(\mathbb{Z}\)-gradings on \(W \xrightarrow{\sim} W^{\vee}\) as does the original integral basis of the ``dual'' coweight lattice (i.e, for \(\widetilde{k} = n + 1 - k\)). In other words, if we denote by \(\Gr_k(a_0, \ldots, a_n)\) and \(\Gr^k(a_0, \ldots, a_n)\) weighted Grassmannians with respect to these integral bases (cf. Notation~\ref{notation:weighted-grassmannian} below), then \(\Gr_k(a_0, \ldots, a_n)\) is isomorphic to \(\Gr^{\widetilde{k}}(a_0, \ldots, a_n)\) with its embedding to a weighted projective space.
\end{remark}

Recall that we have to restrict ourselves to \(\mathbb{Z}_{+}\)-gradings to define a weighted Grassmannian. In our terms the positivity can be described as follows:

\begin{corollary}[see Corollary~\ref{corollary:SL-grading}]
  Let \(\psi \colon G \hookrightarrow \GL(W)\), \(W = \Lambda^k(V)\), be a fundamental representation. The coweights of \(G(\psi)\) endowing the vector space \(W\) with a \(\mathbb{Z}_{+}\)-grading can be identified with elements of the following semi-group:
  \[
    \Gamma(n, k) =
    \left \{
      (a_0, \ldots, a_n) \in \mathbb{Z}^{n + 1} \mid \sum_{j = 1}^k a_{i_j} > \sum_{l = 0}^{k - 2} a_l \quad \forall I = \{i_1 < \cdots < i_k\}
    \right \}.
  \]
  As a consequence, any weighted Grassmannian is isomorphic to one of those defined by \(\sum_{i = 0}^n a_i \gamma^{\vee}_i\), where \((a_0, \ldots, a_n) \in \Gamma(n, k)\), and \(\{\gamma^{\vee}_i\}\) as in Proposition~\ref{proposition:SL-grading}.
\end{corollary}

\begin{remark}
  Our description is applicable to computation of ordinary and equivariant cohomology of weighted Grassmannians in~\cites{abe/equivariant,abe/schur,brahma/cohomology,graham/weighted}.
\end{remark}

\begin{notation}\label{notation:weighted-grassmannian}
  For any \((a_0, \ldots, a_n) \in \Gamma(n, k)\) we denote by \(\Gr_k(a_0, \ldots, a_n)\) the weighted Grassmannian defined by the coweight \(\sum_{i = 0}^n a_i \gamma^{\vee}_i\) of \(G(\psi)\).
\end{notation}

\begin{remark}
  Note that for \(k = 1\) we have \(\Gr_1(a_0, \ldots, a_n) = \mathbb{P}(a_0, \ldots, a_n)\), and if \(a_0 = \cdots = a_n = 1\), then we obtain usual Grassmannians for any \(k\).
\end{remark}

\begin{remark}
  For any \(k > 0\) and \((a_0, \ldots, a_m) \in \Gamma(m, k)\) we have a chain of inclusions
  \[
    \ldots \subset \Gr_k(a_0, \ldots, a_{n - 1}) \subset \Gr_k(a_0, \ldots, a_n) \subset
    \Gr_k(a_0, \ldots, a_{n + 1}) \subset \ldots
  \]
\end{remark}

\begin{remark}
  If \(k > 1\), then in general some of the coefficients \(a_i\) can be \emph{negative}, yet this will yield a \(\mathbb{Z}_{+}\)-grading, provided that \((a_0, \ldots, a_n) \in \Gamma(n, k)\).
\end{remark}

\subsection{Proof of Proposition~\ref{proposition:SL-grading}}

At last, let us prove Proposition~\ref{proposition:SL-grading}.

\begin{lemma}[see~{\cite[Table~2]{onishchik/lie}}]\label{lemma:cartan-inverse}
  Let \(A_n\) be the Cartan matrix of type \(A_n\). Then its inverse \(A_n^{-1}\) can be described as follows:
  \[
    (A_n^{-1})_{ij} = \min(i, j) - \frac{ij}{n + 1}.
  \]
\end{lemma}

\begin{lemma}\label{lemma:SL-lattice description}
  Let \(V \simeq \langle e_0, \ldots, e_n \rangle\), and \(\psi \colon G \hookrightarrow \GL(W)\), \(W = \Lambda^k(V)\), be a fundamental representation. After the identification \(Y(\GL(V)) \simeq \mathbb{Z} \langle e_0, \ldots, e_n \rangle \) the sublattice \(L(\psi) \subset Y(\GL(W))_{\mathbb{Q}}\) defined in Computation~\ref{computation:lattice-calculations} is generated by
  \[
    w_0 = \sum_{i_1 < \cdots < i_k} \binom{n + 1}{k}^{-1} e_{i_1 < \cdots < i_k}, \quad
    w_j =
    \sum_{i_1 < \cdots < i_k}
    \left (
      - \frac{kj}{n + 1} + \vert \{ t : j > i_t \} \vert 
    \right)
    e_{i_1 < \cdots < i_k}.
  \]
\end{lemma}

\begin{proof}
  The (transposed) inverse Cartan matrix of Lemma~\ref{lemma:cartan-inverse} provides a change of coordinates from simple coroots to fundamental coweights. We only have to recall that after the identification \(Y(\GL(V)) \simeq \mathbb{Z} \langle e_0, \ldots, e_n \rangle \) the images of simple coroots \(\{\alpha_i^{\vee}\}\) under the standard representation \(G \hookrightarrow \SL(V)\) are identified with vectors \(e_{i - 1} - e_i\) for all \(i = 1, \ldots, n\).
\end{proof}

\begin{proof}[Proof of Proposition~\ref{proposition:SL-grading}]
  Fix \(k > 0\). We are going to prove the statement by induction on \(l = n + 1 - k > 0\). To this end, let us fix the following notation.

  For any \(l > 0\) we denote by \(V_l = \langle e_0, \ldots, e_n \rangle\) the \((l + k)\)-dimensional vector space with chosen basis. We put \(G_l = \SL(V_l)\), so \(G_l \hookrightarrow \GL(V_l)\) is its standard representation. We also denote by \(\psi_l \colon G_l \hookrightarrow \GL(W_l)\), \(W_l = \Lambda^k(V_l)\), the corresponding fundamental representation of \(G_l\). As before, we introduce the product \(G_l(\psi_l) = G_l \cdot \mathbb{G}_m \subset \GL(W_l)\), where \(\tau^{\vee}_{W_l} \colon \mathbb{G}_m \hookrightarrow \GL(W_l)\), \(\tau^{\vee}_{W_l}(c) = c \cdot \Id_{W_l}\), is the central coweight of \(\GL(W_l)\). We denote by \(Y(G_l)\), \(Y(G_l(\psi_l))\), and \(Y(W_l)\), the coweight lattices of the groups \(G_l\), \(G_l(\psi_l)\), and \(\GL(W_l)\), correspondingly. At last, let
  \[
    M_l \subset Y(G_l(\psi_l))_{\mathbb{Q}} \simeq Y(G_l)_{\mathbb{Q}} \oplus \mathbb{Q} \langle\deg(W_l)^{-1} \cdot \tau^{\vee}_{W_l} \rangle \subset Y(W_l)_{\mathbb{Q}}
  \]
  be the integral submodule generated by columns of the matrix
  \begin{gather*}
    \begin{pmatrix}
      1 & -1 & 0 & \cdots & 0 & 0 & 0 \\ 
      0 & 1 & -1 & \cdots & 0 & 0 & 0 \\
      \cdots & \cdots & \cdots & \cdots & \cdots & \cdots & \cdots \\
      0 & 0 & 0 & \cdots & 1 & -1 & 0 \\
      0 & 0 & 0 & \cdots & 0 & 1 & -1 \\
      \alpha_l & \alpha_l & \alpha_l & \cdots & \beta_l & \beta_l & \beta_l
    \end{pmatrix}, \quad
    \begin{cases}
      \alpha_l = (\frac{k}{l + k} - 1) \binom{l + k}{k}, \\
      \beta_l = \frac{k}{l + k} \binom{l + k}{k} = \binom{l + k - 1}{k - 1}, \\
      (l + 1) \beta_l + (k - 1) \alpha_l = \binom{l + k}{k},
    \end{cases}
  \end{gather*}
  with respect to the basis \((\omega_1^{\vee}, \ldots, \omega_l^{\vee}, \deg(W_l)^{-1} \cdot \tau^{\vee}_{W_l})\), where \(\alpha_l\) is repeated \((k - 1)\) times. Multiplication with the matrix provided by Lemma~\ref{lemma:SL-lattice description} shows that \(M_l\) is a subgroup (of finite index) of \(Y(G_l(\psi_l))\). We denote these generators by \(g_0^l, \ldots, g_n^l\).

  To sum it up, our goal is to prove that \(M_l = Y(G_l(\psi_l))\) and apply Computation~\ref{computation:lattice-calculations}.

  \textbf{Induction basis.} If \(l = 1\), then Lemma~\ref{lemma:SL-lattice description} provides a square non-degenerate matrix whose inverse is the following matrix defined over \(\mathbb{Z}\):
  \[
    \begin{pmatrix}
      -1 & 1 & 0 & \cdots & 0 & 0 & 0 \\ 
      0 & -1 & 1 & \cdots & 0 & 0 & 0 \\
      \cdots \\
      0 & 0 & 0 & \cdots & -1 & 1 & 0 \\
      0 & 0 & 0 & \cdots & 0 & -1 & 1 \\ 
      1 & 1 & 1 & \cdots & 1 & 1 & 1 \\
    \end{pmatrix}.
  \]
  It is not hard to check that this matrix is column-equivalent over \(\mathbb{Z}\) to
  \vspace{0.6cm}
  \[
    \begin{pNiceMatrix}
      1 & -1 & 0 & \cdots & 0 & 0 & 0 \\ 
      0 & 1 & -1 & \cdots & 0 & 0 & 0 \\
      \cdots \\
      0 & 0 & 0 & \cdots & 1 & -1 & 0 \\
      0 & 0 & 0 & \cdots & 0 & 1 & -1 \\
      -1 & -1 & -1 & \cdots & -1 & n & n
      \CodeAfter
      \OverBrace{1-1}{6-5}{k - 1}
      \OverBrace{1-6}{6-7}{2}
    \end{pNiceMatrix},
    \quad n + 1 = k + l.
  \]
  We just multiply the first \((k - 1)\) columns by \((-1)\), transpose the last two columns, and subtract from each of them the sum of the first \((k - 1)\) columns. As a consequence, we obtain that \(M_1 = Y(G_1(\psi_1))\), so we are done.

  \textbf{Induction step.} Let us assume that the statement is proved for some \(l > 0\). We are going to prove it for \(l + 1\).

  To this end, let us consider the inclusion of groups \(i_l \colon G_l \hookrightarrow G_{l + 1}\) and the corresponding inclusion of coweights \(i_l \colon Y(G_l) \hookrightarrow Y(G_{l + 1})\). Both of them by construction extend to the inclusion of groups \(i_l \colon \GL(W_l) \hookrightarrow \GL(W_{l + 1})\) and the corresponding inclusion of coweights \(i_l \colon Y(W_l) \hookrightarrow Y(W_{l + 1})\). Note that the latter inclusion is splittable, where the splitting is provided by a choice of basis on each of \(V_l\), hence on \(W_l\) as well. Then we can correctly define the corestriction map \(\pi_{l + 1} \colon Y(W_{l + 1}) \rightarrow Y(W_l)\), which induces the corestriction map \(\pi_{l + 1} \colon Y(G_{l + 1}(\psi_{l + 1})) \rightarrow Y(G_l(\psi_l))\). To put it simply, \(\pi_{l + 1}\) sends the fundamental coweight \(\omega_{n + 1}^{\vee}\) to \(0\), and the central coweight \(\tau^{\vee}_{W_{l + 1}}\) of \(\GL(W_{l + 1})\) to the central coweight \(\tau^{\vee}_{W_l}\) of \(\GL(W_l)\).
  
  We obtain the following exact sequence of subgroups in the lattice \(Y(W_{l + 1})\):
  \begin{align*}
    0 \rightarrow \ker(\pi_{l + 1}) \cap Y(G_{l + 1}(\psi_{l + 1}))  \rightarrow Y(G_{l + 1}(\psi_{l + 1})) \xrightarrow{\pi_{l + 1}} \\ \xrightarrow{\pi_{l + 1}} Y(G_l(\psi_l)) \rightarrow \coker(\pi_{l + 1}) \cap Y(G_l(\psi_l)) \rightarrow 0.
  \end{align*}
  It can be easily seen that \(\pi_{l + 1}(g_i^{l + 1}) = g_i^l\) for any \(i = 0, \ldots, n\). Consequently, \(\pi_{l + 1}(M_{l + 1}) = M_l\), which equals to \(Y(G_l(\psi_l))\) by the induction hypothesis, hence the cokernel \(\coker(\pi_{l + 1}) \cap Y(G_l(\psi_l))\) is trivial. Moreover, we can easily check that
  \[
    (l + 1 + k) \cdot \omega_{l + 1}^{\vee} = k \cdot \left (\sum_{i = 0}^n g_i^{l + 1} \right )
    - (l + 1) \cdot g_{n + 1}^{l + 1} \in
    M_{l + 1} \subset Y(G_{l + 1}(\psi_{l + 1})).
  \]
  Note that \((l + 1 + k) / \gcd(l + 1, k) \cdot \omega_{l + 1}^{\vee}\) is a smallest integral multiple of \(\omega_{l + 1}^{\vee}\) (see Lemma~\ref{lemma:SL-lattice description}). Then the (one-dimensional) kernel \(\ker(\pi_{l + 1}) \cap Y(G_{l + 1}(\psi_{l + 1}))\) is generated by the element \((l + 1 + k) / \gcd(l + 1, k) \cdot \omega_{l + 1}^{\vee} \in M_{l + 1}\). We conclude that \(M_{l + 1} = Y(G_{l + 1}(\psi_{l + 1}))\), so we are done.
\end{proof}

\begin{question}\label{question:grading}
  It seems that the proof of Proposition~\ref{proposition:SL-grading} can be generalised to other root systems as well, as long as we are provided with the explicit description of the image of simple coroots under the chosen fundamental representation. Actually, the explicit inverse of the Cartan matrix for any irreducible root system can be found in~\cite[Table~2]{onishchik/lie}. Yet, our proof depends heavily on the existence of a ``nice'' basis for each \(M_l = Y(G_l(\psi_l))\). Is it possible to find a similar basis for other root systems?
\end{question}

\section{Geometric invariants of weighted Grassmannians}\label{section:series-A-invariants}

In this section we are going to describe the dualising sheaf and the Hilbert series in terms of our parametrisation of \(\mathbb{Z}\)-gradings. We continue to follow the notations of Section~\ref{section:series-A-grading}. We start with the criterion for well-formedness of a weighted Grassmannian.

\begin{lemma}\label{lemma:well-formedness}
  Let \(Y = \Gr_k(a_0, \ldots, a_n) \subset \mathrm{w} \mathbb{P}\), and \((\deg(T_{i_1 < \cdots < i_k}))\) be the values of its \(\mathbb{Z}_{+}\)-grading (taken with repetitions), where
  \[
    \deg(T_{i_1 < \cdots < i_k}) = -\sum_{l = 0}^{k - 2} a_l + \sum_{j = 1}^k a_{i_j}.
  \]
  Then \(Y\) is well-formed if and only \(\mathrm{w} \mathbb{P}\) is, i.e., if after the removing from this tuple any element the greatest common divisor of the remaining ones equals 1.
\end{lemma}

\begin{proof}
  Assume that \(\mathrm{w} \mathbb{P}\) is a well-formed weighted projective space. It is well-known (for example, see~\cite[\nopp 5.15]{ianofletcher/weighted}) that for any well-formed weighted projective space \(\mathrm{w} \mathbb{P} = \mathbb{P}(b_0, \ldots, b_N) = \Proj(R)\) its singular locus is the union of strata
  \[
    \Sing(\mathrm{w} \mathbb{P}) = \bigcup_I \{T'_i = 0 \mid i \not \in I\}, \quad
    R = \Bbbk[T'_0 ,\ldots, T'_N], \quad \deg(T'_i) = b_i.
  \]
  over all subsets \(I\) of \(\{0, \ldots, N\}\) such that \(\gcd(\{b_i \mid i \in I\}) > 1\).

  Let us present \(Y \subset \mathrm{w} \mathbb{P}\) as a quotient \(\widetilde{Y} \rightarrow Y\) of a usual Grassmannian \(\widetilde{Y} \subset \mathbb{P}\) induced by the quotient map \(\mathbb{P} \rightarrow \mathrm{w} \mathbb{P}\)  (see Remark~\ref{remark:quotient}). We denote by \(\{\widetilde{T}_I\}\) and \(\{T_I\}\) the Pl\"{u}cker coordinates on \(\mathbb{P}\) and \(\mathrm{w} \mathbb{P}\), respectively. In these terms we have to impose \(\gcd(\{\deg(a_I) \mid I \not \in M\}) = 1\) for any subset \(M \subset \{I = (i_1, \ldots, i_k)\}\) such that the equations \(\{\widetilde{T}_I = 0, \; I \in M\}\) define a Schubert divisor on the Grassmannian \(\widetilde{Y} \subset \mathbb{P}\). But this can happen if and only if \(\vert M \vert = 1\) (for example, see~\cite[Chapter~4]{lakshmibai/monomial}). In other words, we have to impose \(\gcd(\{\deg(a_I) \mid I \not \in M\}) = 1\) whenever \(\vert M \vert = 1\), but this is the definition of a well-formed weighted projective space.
\end{proof}

\subsection{Dualising sheaf}

\begin{proposition}\label{proposition:dualising-sheaf}
  Let \(Y = \Gr_k(a_0, \ldots, a_n) \subset \mathrm{w} \mathbb{P}\) be a well-formed weighted Grassmannian. The dualising sheaves of \(Y\) and \(\mathrm{w} \mathbb{P}\) are isomorphic to
  \begin{gather*}
    \omega_Y \simeq \mathcal{O}_Y
    \left (
      - k
      \left (
        \sum_{i = 0}^n a_i
      \right ) +
      (n + 1)
      \left (
        \sum_{i = 0}^{k - 2} a_i
      \right )
    \right), \\
    \omega_{\mathrm{w} \mathbb{P}} \simeq \mathcal{O}_{\mathrm{w} \mathbb{P}}
    \left (
      - \binom{n}{k - 1}
      \left (
        \sum_{i = 0}^n  a_i
      \right ) +
      \binom{n + 1}{k}
      \left (
        \sum_{i = 0}^{k - 2} a_i
      \right )
    \right ).
  \end{gather*}
\end{proposition}

\begin{example}
  For \(k = 2\), \(n = 3\), the single Pl\"{u}cker relation is given by
  \begin{gather*}
    T_{0, 1} T_{2, 3} + T_{1, 2} T_{0, 3} - T_{0, 2} T_{1, 3} = 0.
  \end{gather*}
  We can clearly see that it is quasi-homogeneous of the degree
  \[
    -a_0 + a_1 + a_2 + a_3 = \deg(-K_{\mathrm{w} \mathbb{P}}) - \deg(-K_Y),
  \]
  which is just the adjunction formula for a weighted hypersurface \(Y \subset \mathrm{w} \mathbb{P}\).
\end{example}

\begin{remark}
  Note that for \(k = 1\) or \((a_0, \ldots, a_n) = (1, \ldots, 1)\) we obtain the classical formulae for weighted projective spaces and for usual Grassmannians.
\end{remark}

\begin{proof}[Proof of Proposition~\ref{proposition:dualising-sheaf}]
  We can present \(Y \subset \mathrm{w} \mathbb{P}\) as a quotient \(\pi \colon \widetilde{Y} \rightarrow Y\) of a usual Grassmannian \(\widetilde{Y} \subset \mathbb{P}\) induced by the quotient map \(\pi \colon \mathbb{P} \rightarrow \mathrm{w} \mathbb{P}\) (see Remark~\ref{remark:quotient}), where the finite abelian group \(\mathcal{G}\) acts on coordinates of \(\mathbb{P}\) by multiplication on roots of unity. We denote by \(\{\widetilde{T}_I\}\) and \(\{T_I\}\) the Pl\"{u}cker coordinates on \(\mathbb{P}\) and \(\mathrm{w} \mathbb{P}\), respectively. Note that \(\Sing(Y) = Y \cap \Sing(\mathrm{w} \mathbb{P})\) by Lemma~\ref{lemma:singular-locus}

  The dualising sheaf \(\omega_Y\) satisfies the \((S2)\)-property by~\cite[Corollary~5.69]{kollar/birational}. Lemma~\ref{lemma:serre-sheaf-extension} says that we have
  \(\mathcal{O}_Y(k) \simeq j_* j^* \mathcal{O}_Y(k) \simeq j_* j^* i^* \mathcal{O}_{\mathrm{w} \mathbb{P}}(k)\)
  for any \(k\), where \(j \colon Y^{\circ} \hookrightarrow Y\) and \(i \colon Y \hookrightarrow \mathrm{w} \mathbb{P}\) are natural embeddings. In other words, we only have to prove the isomorphism along the smooth open part \(Y^{\circ} = Y \setminus \Sing(Y) = Y \setminus \Sing(\mathrm{w} \mathbb{P})\).

  Let us also recall that along the smooth open part the dualising sheaf is isomorphic to the usual canonical sheaf (see~\cite[Theorem~6.4.32]{liu/algebraic}). Moreover, it is known that \(\omega_Y = \pi_* (\omega_{\widetilde{Y}}^{\mathcal{G}})\) (for example, see~\cite{peskin/dualizing}).  In other words, it is sufficient to explicitly construct a \(\mathcal{G}\)-invariant rational differential form on \(\widetilde{Y}^{\circ} = \pi^{-1}(Y^{\circ})\).

  Note that the tori \(T_{G(\psi)} / \im(\tau_W^{\vee})\) and \(T_{G(\psi)} / \im(\gamma^{\vee})\) act on \(\widetilde{Y}\) and \(Y\), respectively. Moreover, \(Y = \widehat{Y} // \im(\gamma^{\vee}) \simeq \widetilde{Y} / \mathcal{G}\), where \(\gamma^{\vee}\) acts on \(\widetilde{Y}\) through \(T_{G(\psi)} / \im(\tau_W^{\vee})\). Recall that the Grassmannian \(\widetilde{Y}\) admits a \(T_{G(\psi)}\)-invariant rational differential form with the divisor on~\(\widetilde{Y}\) of the following form:
  \[
    -\sum_{\sigma \in \langle (0, \ldots, n) \rangle}
    \widetilde{D}_{\sigma(\{0, \ldots, k - 1\})}, \quad
    \widetilde{D}_I = \{\widetilde{T}_I = 0\} \cap \widetilde{Y} \subset \widetilde{Y}, \quad
    \langle (0, \ldots, n) \rangle \simeq \mathbb{Z} / (n + 1) \mathbb{Z}.
  \]
  This rational differential form correctly defines a rational differential form on \(Y^{\circ}\). We will denote by \(K_Y\) the corresponding divisor on \(Y\).

  Now we only have to compute the degrees by means of Proposition~\ref{proposition:SL-grading}:
  \[
    \deg(-K_Y) = \sum_{\sigma \in \langle (0, \ldots, n) \rangle} \deg(T_{\sigma(\{0, \ldots, k - 1\})}), \quad
    \deg(-K_{\mathrm{w} \mathbb{P}}) = \sum_{I = \langle i_1, \ldots, i_k \rangle} \deg(T_I). \qedhere
  \]
\end{proof}

\subsection{Hilbert series}

To any weighted generalised Grassmannian \(Y \subset \mathrm{w} \mathbb{P}\) we can correspond its \emph{Hilbert series} (see Definition~\ref{definition:hilbert-series}). We can explicitly compute the Hilbert series of \(Y \subset \mathrm{w} \mathbb{P}\) with the following version of the Weyl character formula.

\begin{proposition}[see~{\cite[Theorem~3.2]{qureshi/flag-I}}]\label{proposition:hilbert-series-weyl}
  Let \(Y \subset \mathrm{w} \mathbb{P}\) be a weighted generalised Grassmannian defined by a fundamental weight \(\lambda\) and a positive coweight of form \(\gamma^{\vee} = \psi(\mu^{\vee}) + d \cdot \tau^{\vee}_{V_{\lambda}}\), where \(\psi \colon G \hookrightarrow \GL (V_{\lambda})\) is the corresponding representation of the highest weight \(\lambda\), and \(\mu^{\vee} \in Y(G)\) is an integral coweight of \(G\).

  Then the Hilbert series of \(Y \subset \mathrm{w} \mathbb{P}\) has the following form:
  \begin{gather*}
    \mathcal{H}(Y) = \frac{\widetilde{\mathcal{P}}(Y)}{\widetilde{\mathcal{Q}}(Y)}, \quad
    \widetilde{\mathcal{P}}(Y) = \sum_{\sigma \in \mathcal{W}} (-1)^{\sigma}
    \frac{t^{\langle \sigma \rho, \mu^{\vee} \rangle}}{1 - t^{\langle \sigma \lambda, \mu^{\vee} \rangle + d}}, \\
    \widetilde{\mathcal{Q}}(Y) = \sum_{\sigma \in \mathcal{W}} (-1)^\sigma t^{\langle \sigma \rho, \mu^{\vee} \rangle} =
    t^{\langle \rho, \mu^{\vee} \rangle} \prod_{\alpha \in \Phi_+} (1 - t^{\langle - \alpha, \mu^{\vee} \rangle}),
  \end{gather*}
  where \(\mathcal{W}\) is the Weyl group, \(\alpha \in \Phi_+\) are positive roots, and \(\rho\) is the Weyl vector.
\end{proposition}

\begin{remark}
  It would be more convenient for us to write this formal series as
  \[
    \mathcal{H}(Y) = \frac{\mathcal{P}(Y)}{\mathcal{Q}(Y)}, \quad \mathcal{P}(Y) = \sum_{\sigma \in \mathcal{W}} (-1)^\sigma
    \frac{t^{\langle \sigma \rho - \rho, \mu^{\vee} \rangle}}{1 - t^{\langle \sigma \lambda, \mu^{\vee} \rangle + d}}, \quad
    \mathcal{Q}(Y) = \prod_{\alpha \in \Delta_+} (1 - t^{\langle - \alpha, \mu^{\vee} \rangle}).
  \]
  This way \(\langle \sigma \rho - \rho, \omega_i^{\vee} \rangle\) is always integral for any fundamental coweight \(\omega_i^{\vee}\) of \(G\).
\end{remark}

From now on \(G = \SL(V)\), and \(Y = \Gr_k(a_0, \ldots, a_n)\) is a weighted Grassmannian.

\begin{proposition}\label{proposition:hilbert-series}
  Let \(Y = \Gr_k(a_0, \ldots, a_n)\) be a weighted Grassmannian. Its Hilbert series (see Definition~\ref{definition:hilbert-series}) equals \(\mathcal{H}(Y) = \mathcal{P}(Y) / \mathcal{Q}(Y)\), where
  \begin{gather*}
    \mathcal{P}(Y) = \sum_{I = (i_1 < \cdots < i_k)} \frac{F(I)}{1 - t^{a(I)}}, \quad
    \mathcal{Q}(Y) = \prod_{\substack{i,j = 0, \ldots, n \\ i < j}} (1 - t^{a_j - a_i}); \\
    a(I) = -\sum_{l = 0}^{k - 2} a_l + \sum_{j = 1}^k a_{i_j} > 0, \quad
    F(I) =
    \left (
      \sum_{\sigma \in S_I} (-1)^{\sigma} t^{f(\sigma)}
    \right ); \\
    S_I = \{\sigma \in S_{n + 1} : \sigma(\{0, \ldots, k - 1\}) = I\}, \quad
    f(\sigma) = \sum_{i = 0}^n (i - \sigma^{-1}(i)) a_i.
  \end{gather*}
\end{proposition}

\begin{remark}
  Note that \(S_I = \sigma_I \cdot \Stab(I)\), where \(\Stab(I) \simeq S_k \times S_{n + 1 - k}\) preserves the partition of \(\{0, \ldots, n\}\) into \(I\) and its complement, and \(\sigma_I \in S_{n + 1}\) is a permutation sending \(l\) to \(i_{l + 1}\) for any \(l = 0, \ldots, k - 1\).
\end{remark}

\begin{example}
  Let \(k = 1\) and \(n = 2\), i.e., \(Y = \mathbb{P}(a_0, a_1, a_2)\). Then we have
  \begin{gather*}
    \mathcal{P}(Y) =
    \frac{1 - t^{-a_1 + a_2}}{1 - t^{a_0}} +
    \frac{-t^{-a_0 + a_1} + t^{-2 a_0 + a_1 + a_2}}{1 - t^{a_1}} +
    \frac{-t^{-2 a_0 + 2 a_2} + t^{-a_0 - a_1 + 2 a_2}}{1 - t^{a_2}}, \\
    \mathcal{Q}(Y) = (1 - t^{a_2 - a_1}) (1 - t^{a_2 - a_0}) (1 - t^{a_1 - a_0}).
  \end{gather*}
  One can check by direct computations that \(\mathcal{P}(Y) / \mathcal{Q}(Y) = \prod_{i = 0}^2 (1 - t^{a_i})^{-1}\).
\end{example}

\begin{example}
  Let \(k = 2\) and \(n = 4\), i.e., \(Y = \Gr_2(a_0, \ldots, a_4)\) is a weighted Grassmannian of lines in a five-dimensional vector space. Then we have
  \begin{gather*}
    \mathcal{P}(Y) = \sum_{i = 1}^4 \frac{F{(0, i)}}{1 - t^{a_i}} + 
    \sum_{\substack{i,j = 1, \ldots, 4 \\ i < j}} \frac{F(i,j)}{1 - t^{a_i + a_j - a_0}}, \quad
    \mathcal{Q}(Y) = \prod_{\substack{i,j = 0, \ldots, 4 \\ i < j}} (1 - t^{a_j - a_i}); \\    
    F(i, j) =
    \left (
      \sum_{\sigma \in S(i, j)} (-1)^{\sigma} t^{f(\sigma)}
    \right ); \quad
    f(\sigma) = \sum_{i = 0}^4 (i - \sigma^{-1}(i)) a_i; \\
    S(i, j) = \{\sigma \in S_5 : \sigma(\{0, 1\}) = \{i, j\}\}.
  \end{gather*}
  It can be checked by direct computation (see Remark~\ref{remark:GL-construction}) that our formula is compatible with the formula for Hilbert series for \(Y = \Gr_2(a_0, \ldots, a_4)\) in~\cite{corti/weighted}.
\end{example}

\begin{remark}
  Hilbert--Serre theorem implies that the Hilbert series of any weighted Grassmannian can always be written in the form \(\mathcal{H}(Y) = P(t) / \prod_I (1 - t^{\deg(T_I)})\). There is a lot of a cancellation involved in reducing the Hilbert series of a weighted Grassmannian given by Proposition~\ref{proposition:hilbert-series} to this compact form.
\end{remark}

\begin{problem}
  Find a closed comprehensive formula for the Hilbert polynomial \(\mathcal{H}(Y) \cdot \prod_I (1 - t^{\deg(T_I)})\) of a weighted Grassmannian \(Y = \Gr_k(a_0, \ldots, a_n)\).
\end{problem}

To prove Proposition~\ref{proposition:hilbert-series}, we only have to combine Propositions~\ref{proposition:SL-grading} and~\ref{proposition:hilbert-series-weyl}. 

\begin{lemma}
  The denominator \(\mathcal{Q}(Y)\) has the form
  \[
    \mathcal{Q}(Y) = \prod_{\substack{i,j = 0, \ldots, n \\ i < j}} (1 - t^{a_j - a_i}).
  \]
\end{lemma}

\begin{proof}
  Follows directly from the parametrisation provided by Proposition~\ref{proposition:SL-grading} and the classical description of positive roots for \(\SL(V)\).
\end{proof}
  
\begin{lemma}\label{lemma:weyl-orbit}
  For any \(k = 1, \ldots, n\) let \(\omega_k\) be the fundamental weight of the group \(G = \SL(V)\). Then the Weyl orbit \(\mathcal{W} \cdot \omega_k\) consists of the following elements:
  \[
    \lambda_I = \sum_{j = 1}^k e_{i_j}, \quad I = (i_1 < \cdots < i_k).
  \]
  In particular, we have \(\omega_k = \lambda_I\), where we put \(I = \{0, \ldots, k - 1\}\). 
\end{lemma}

\begin{proof}
  The Weyl group \(\mathcal{W}\) is generated by simple reflections \(r_1, \ldots, r_n\) acting on fundamental weights \(\omega_i\) as follows:
  \[
    r_i (\omega_j) =
    \begin{cases}
      \omega_j, & \text{ if } i \neq j; \\
      \omega_i - \alpha_i, & \text{ otherwise.}
    \end{cases}
  \]
  From this description we see that \(\mathcal{W} \simeq S_{n + 1}\) sends \(e_i\) to \(e_{\sigma(i)}\), and \(\lambda_I\) to \(\lambda_{\sigma(I)}\).
\end{proof}

\begin{corollary}
  For any \(I = (i_1 < \cdots < i_k)\) the stabiliser \(\Stab(\lambda_I)\) is isomorphic to \(S_k \times S_{n + 1 - k}\) and consists of permutations on \(\{0, \ldots, n\}\) preserving the partition of \(\{0, \ldots, n\}\) into the subset \(I \subset \{0, \ldots, n\}\) and its complement.
\end{corollary}

\begin{lemma}
  The numerator \(\mathcal{P}(Y)\) can be presented as
  \begin{gather*}
    \mathcal{P}(Y) = \sum_{I = (i_1 < \cdots < i_k)} \frac{F(I)}{1 - t^{a(I)}}, \quad
    F(I) =
    \left (
      \sum_{\sigma \in S_I} (-1)^{\sigma} t^{f(\sigma)}
    \right ), \\
    a(I) = -\sum_{l = 0}^{k - 2} a_l + \sum_{j = 1}^k a_{i_j}, \quad
    f(\sigma) = \sum_{i = 0}^n (i - \sigma^{-1}(i)) a_i, \\
    S_I = \{\sigma \in S_{n + 1} : \sigma(\{0, \ldots, k - 1\}) = I\}, \quad
  \end{gather*}
\end{lemma}

\begin{proof}
  We can rewrite \(f(\sigma)\) as follows:
  \[
    f(\sigma) = \langle \sigma \rho, \mu^{\vee} \rangle - \langle \rho, \mu^{\vee} \rangle =
    \langle \rho, \sigma^{-1}(\mu^{\vee}) \rangle - \langle \rho, \mu^{\vee} \rangle.
  \]
  
  We present the Weyl vector as a (rational) linear combination of simple roots:
  \[
    \rho = \sum_{l = 1}^n \omega_l = \sum_{l = 1}^n c_l \alpha_l, \quad
    c_l = \sum_{j = 1}^n \min(l, j) - \frac{l n}{2}.
  \]
  Note that \(c_0 = c_{n + 1} = 0\), and \(c_{l + 1} - c_l = (n - l) - n/2\).

  Let us present the action of the Weyl group on the coweight \(\mu^{\vee}\) as follows:
  \[  
    \sigma^{-1}(\mu^{\vee}) =
    \sigma^{-1} \left ( \sum_{i = 0}^n a_i e_i^{\vee} \right ) =
    \sum_{i = 0}^n a_i e_{\sigma^{-1}(i)}^{\vee} =
    \sum_{i = 0}^n a_{\sigma(i)} e^{\vee}_i = \sum_{j = 1}^n (a_{\sigma(j - 1)} - a_{\sigma(j)}) \omega_{j}^{\vee}.
  \]
  Consequently, we obtain from Proposition~\ref{proposition:SL-grading} that
  \begin{gather*}
    f(\sigma) =
    \sum_{j = 1}^n c_j (a_{\sigma(j - 1)} - a_{\sigma(j)}) -
    \sum_{j = 1}^{n - 1} c_j (a_{j - 1} - a_j) = \\
    \sum_{i = 0}^n (c_{\sigma^{-1}(i) + 1} - c_{\sigma^{-1}(i)}) a_i
    - \sum_{i = 0}^n (c_{i + 1} - c_i) a_i =
    \sum_{i = 0}^n (i - \sigma^{-1}(i)) a_i. \qedhere
  \end{gather*}
\end{proof}

\begin{proof}[Proof of Proposition~\ref{proposition:hilbert-series}]
  Apply Proposition~\ref{proposition:hilbert-series-weyl} to the above-computed terms.
\end{proof}

\clearpage
\printbibliography

\end{document}